\documentclass[smallextended]{svjour3}
\smartqed

\usepackage{graphicx,fix-cm,amssymb,amsmath,hyperref,booktabs,multirow,bbm}

\newcommand{\R}{\mathbb{R}}
\newcommand{\N}{\mathbb{N}}
\newcommand{\Z}{\mathbb{Z}}
\newcommand{\Ccal}{\mathcal{C}}

\newcommand{\vol}{\mathrm{vol}}

\newcommand{\cone}{\mathrm{cone}}

\journalname{}

\smartqed

\begin{document}

\title{A semidefinite programming hierarchy for packing problems in
discrete geometry\thanks{The authors were supported by Vidi grant
639.032.917 from the Netherlands Organization for Scientific
Research (NWO).}
}

\author{David de Laat \and Frank Vallentin}

\institute{D.~de Laat \at
Delft Institute of Applied Mathematics, Delft
University of Technology, P.O. Box 5031, 2600 GA Delft, The
Netherlands
\email{mail@daviddelaat.nl}           
\and
F.~Vallentin \at
Mathematisches Institut, Universit\"at zu
K\"oln, Weyertal 86--90, 50931 K\"oln, Germany
\email{frank.vallentin@uni-koeln.de}
}

\date{Received: date / Accepted: date}

\maketitle

\begin{abstract}
Packing problems in discrete geometry can be modeled as finding
  independent sets in infinite graphs where one is interested in
  independent sets which are as large as possible.  For finite graphs
  one popular way to compute upper bounds for the maximal size of an
  independent set is to use Lasserre's semidefinite programming
  hierarchy. We generalize this approach to infinite graphs. For this
  we introduce topological packing graphs as an abstraction for
  infinite graphs coming from packing problems in discrete
  geometry. We show that our hierarchy converges to the independence
  number.
\end{abstract}

\section{Introduction}

\subsection{Packing problems in discrete geometry}
\label{ssec:packing problems in discrete geometry}

Many, often notoriously difficult, problems in discrete geometry can
be modeled as packing problems in graphs where the vertex set is an
uncountable set having additional geometric structure. 

The most famous
example is the sphere packing problem in $3$-dimensional space, the
Kepler problem, which was solved by Hales \cite{Hales2005a} in
1998. Here the vertex set is~$\R^3$ and two points are adjacent
whenever their Euclidean distance is in the open interval $(0,2)$.

An \emph{independent set} of an undirected graph $G = (V, E)$ is a
subset of the vertex set which does not span an edge. In the sphere
packing case, an independent set corresponds to centers of unit balls
which do not intersect in their interior. Now one is trying to find an
independent set which covers as much space as possible. What ``much''
means depends on the situation. When the vertex set $V$, the
\emph{container}, is compact and when we pack identical shapes we can
simply count and we use the \emph{independence number}
\[
\alpha(G) = \sup\{|I| : I \subseteq V,\, I \text{ is independent}\}.
\]

If the objects are of different size we provide them with a weight
$w(x)$ and we use the \emph{weighted independence number}
\[
\alpha_w(G) = \sup\Big\{\sum_{x \in I} w(x) : I \subseteq V,\, I \text{ is independent}\Big\}.
\]
In the non-compact sphere packing case one needs to use a density
version of the independence number since maximal independent sets have
infinite cardinality: The \emph{(upper) point density} of an independent set
$I \subset \mathbb{R}^3$ is
\[
\delta(I) = \limsup_{R \to \infty} \frac{|I \cap [-R,R]^3|}{\vol([-R,R]^3)},
\]
where $[-R,R]^3$ is the cube centered at the origin with side length $2R$.
This measures the number of centers of unit balls per unit volume. To
determine the geometric density of the corresponding sphere packing we
multiply $\delta(I)$ by the volume of the unit ball.

\medskip

More examples include:
\begin{itemize}
\item[---] \emph{Error correcting $q$-ary codes}:
$V = \mathbb{F}_q^n$, where $\{x,y\} \in E$ if their
Hamming distance lies in the open interval $(0,d)$. If $q = 2$ we
speak about \emph{binary codes} and if we restrict to all code words
having the same Hamming norm we speak about \emph{constant weight codes}.\\[-1.7ex]
\item[---] \emph{Spherical codes:} $V = S^{n-1}$,  where $\{x,y\} \in E$
  if their inner product lies in the open interval $(\cos(\theta), 1)$.\\[-1.7ex]
\item[---]  \emph{Codes in real projective space:} $V = \R\mathrm{P}^{n-1}$, where $\{x,y\} \in E$
  if their distance lies in the open interval $(0, d)$.\\[-1.7ex]
\item[---] \emph{Sphere packings:} $V = \R^n$, where $\{x,y\} \in E$ if
  their Euclidean distance lies in the open interval $(0, 2)$.\\[-1.7ex]
\item[---] \emph{Binary sphere packings:} $V = \R^n \times \{1,2\}$ where
  $\{(x,i), (y,j)\} \in E$ if the Euclidean distance between $x$ and
  $y$ lies in the open interval $(0,r_i + r_j)$ and
  $w(x,i) = r_i^n \vol B_n$, where $B_n$ is the unit ball. \\[-1.7ex]
\item[---] \emph{Binary spherical cap packings:} $V = S^{n-1} \times \{1, 2\}$
where $\{(x,i), (y,j)\} \in E$ if the inner product of $x$ and $y$ lies in the
open interval $(\cos(\theta_i + \theta_j), 1)$ and $w(x, i)$ is the volume of the
spherical cap $\{z \in S^{n-1} : x \cdot z \geq \cos(\theta_i) \}$.\\[-1.7ex]
\item[---] \emph{Packings of congruent copies of a convex body:} $V =
  \R^n \rtimes \mathrm{SO}(n)$ where $(x,A)$ and $(y,B)$ are adjacent if $x +
  A\mathcal{K}^\circ \cap y + B \mathcal{K}^\circ \neq \emptyset$, where
  $\mathcal{K}^\circ$ is the interior of the convex body $\mathcal{K}$.
\end{itemize}

Currently, these problems have been solved in only a few special
cases. One might expect that they will never be solved in full
generality, for all parameters. Finding good lower bounds by
constructions and good upper bounds by obstructions are both
challenging tasks. Over the last years the best known results were
achieved with computer assistance: Algorithms like the adaptive
shrinking cell scheme of Torquato and Jiao~\cite{Torquato2009a}
generate dense packings and give very good lower bounds. The
combination of semidefinite programming and harmonic analysis often
gives the best known upper bounds for these packing problems. This
method originated from work of Hoffman~\cite{Hoffman1970a},
Delsarte~\cite{Delsarte1973a}, and Lov\'asz~\cite{Lovasz1979a}.

\subsection{Lasserre's hierarchy for finite graphs}

Computing the independence number of a finite graph is an
$\mathrm{NP}$-hard problem as shown by Karp~\cite{Karp}. Approximating
optimal solutions of $\mathrm{NP}$-hard problems in combinatorial
optimization with the help of linear and semidefinite optimization is
a very wide and active area of research. The most popular semidefinite
programming hierarchies for $\mathrm{NP}$-hard combinatorial
optimization problems are the Lov\'asz-Schrijver hierarchy
\cite{Lovasz1991a} (the $N^+$-operator) and the hierarchy of
Lasserre~\cite{Lasserre2002a}. Laurent~\cite{Laurent2003a} showed that
Lasserre's hierarchy is stronger than the Lov\'asz-Schrijver
hierarchy.

We now give a formulation of Lasserre's hierarchy for computing the
independence number of a finite graph~$G = (V, E)$. Here we follow
Laurent~\cite{Laurent2003a}. The \emph{$t$-th step of Lasserre's
  hierarchy} is:
\[
\mathrm{las}_t(G)=\max\Big\{ \sum_{x \in V} y_{\{x\}} :
y \in \R^{I_{2t}}_{\geq 0}, \; y_\emptyset = 1, \; M_t(y) \text{ is
  positive semidefinite}
\Big\},
\]
where $I_{t}$ is the set of all independent sets with at most
$t$ elements and where $M_t(y) \in \R^{I_t \times I_t}$ is the \emph{moment
matrix} defined by the vector $y$: Its $(J,J')$-entry equals
\[
\left(M_t(y)\right)_{J, J'} = 
\begin{cases}
y_{J \cup J'} & \text{if } J \cup J' \in I_{2t},\\
0 & \text{otherwise.}
\end{cases}
\]

The first step in Lasserre's hierarchy coincides with
the $\vartheta'$-number, the strengthened version of Lov\'asz
$\vartheta$-number \cite{Lovasz1979a} which is due to Schrijver
\cite{Schrijver1979a}; for a proof see for instance the book by Schrijver \cite[Theorem
67.11]{Schrijver2003a}. Furthermore the hierarchy converges to
$\alpha(G)$ after at most $\alpha(G)$ steps:
\[
\vartheta'(G) = \mathrm{las}_1(G)\geq \mathrm{las}_2(G) \geq
\ldots \geq \mathrm{las}_{\alpha(G)}(G) = \alpha(G).
\]

Lasserre \cite{Lasserre2002a} showed this convergence in the general
setting of hierarchies for $0$/$1$ polynomial optimization problems by
using Putinar's Positivstellensatz~\cite{Putinar1993a}. Laurent
\cite{Laurent2003a} gave an elementary proof, which we discuss in
Section~\ref{sec:finite convergence}.

Many variations are possible to set up a semidefinite programming
hierarchy: For instance one can consider only ``interesting''
principal submatrices to simplify the computation and one can also add
more constraints coming from problem specific arguments. In fact, in
the definition of $\mathrm{las}_t(G)$ we used the nonnegativity
constraints $y_S \geq 0$ for $S \in I_{2t}$. Even without them, the
convergence result holds, and the first step in the hierarchy
coincides with the Lov\'asz $\vartheta$-number.

A rough classification for all these variations can be given in terms
of \textit{$n$-point bounds}. This refers to all variations which make
use of variables $y_S$ with $|S| \leq n$. An $n$-point bound is 
capable of using obstructions coming from the local interaction of
configurations having at most $n$ points. For instance the Lov\'asz
$\vartheta$-number is a $2$-point bound and the $t$-th step in
Lasserre's hierarchy is a $2t$-point bound. The relation between
$n$-point bounds and Lasserre's hierarchy was first made explicit by
Laurent~\cite{Laurent2007a} in the case of bounds for binary codes.

\subsection{Topological packing graphs}
\label{ssec:topological packing graphs}

The aim of this paper is to define and analyze a semidefinite
programming hierarchy which upper bounds the independence number for
infinite graphs arising from packing problems in discrete
geometry. For this we consider graphs where vertices which are close
are adjacent, and where vertices which are adjacent will stay adjacent
after small enough perturbations. These two conditions will be essential at
many places in this paper. We formalize them by the following
definition.

\begin{definition}
\label{def:topological packing graph}
A graph whose vertex set is a Hausdorff topological space is called a
\emph{topological packing graph} if each finite clique is contained in
an open clique. An \emph{open clique} is an open subset of the vertex set where every two vertices are adjacent.
\end{definition}

It clearly suffices to verify the condition for cliques of size one and
two.

Of course, every graph is a packing graph when we endow the vertex set
with the discrete topology. However, weaker topologies give stronger
conditions on the edge sets. For instance, when the vertex set of a
topological packing graph is compact, then the independence number is
finite because every single vertex is a clique.

A \emph{distance graph} $G=(V,E)$ is a graph where $(V,d)$ is a metric
space, and where there exists $D \subseteq (0, \infty)$ such that $x$
and $y$ are adjacent precisely when $d(x, y) \in D$. If $D$ is open
and contains the interval $(0, \delta)$ for some $\delta > 0$, then
$G$ is a topological packing graph. That $D$ contains an interval
starting from $0$ implies that vertices which are close are adjacent,
and that $D$ is open implies that adjacent vertices will stay adjacent
after small enough perturbations. The binary spherical cap packing graph as
defined in Section~\ref{ssec:packing problems in discrete geometry} is
a compact topological packing graph with the usual topology on the
vertex set $S^{n-1} \times \{1,2\}$. And although there exists a
metric compatible with this topology which gives the graph as a
distance graph\footnote{Assume $\theta_1 < \theta_2$ and let
  $\epsilon$ be some number strictly between $(1 -
  \theta_1/\theta_2)/2$ and $1$. Let $D = (0, 1)$, and let $d((x,i),
  (y,j))$ be given by $\epsilon \delta_{i \neq j} + (1 - \epsilon
  \delta_{i \neq j}) \arccos(x \cdot y)\, (\theta_1+\theta_2)^{-1}$
  when $x \cdot y < \cos(\theta_i + \theta_j)$ and $1$ otherwise.}, it
is easier and more natural to work directly with the topological
packing graph structure.

Notice that in Definition~\ref{def:topological packing graph}
requiring all cliques to be contained in an open clique --- which by Zorn's lemma is
equivalent to all maximal cliques being open --- would give a strictly
stronger condition.\footnote{Consider the graph with vertex set $[0,
  1] \times \mathbb{Z}$ where $(x, i)$ and $(y, j)$ are adjacent if $i
  = j$ or when $x$ and $y$ are both strictly smaller than $|i-j|^{-1}$
  (for $i \neq j$). Here each finite clique is contained in an open clique, 
  but the countable clique $\{0\} \times \mathbb{Z}$ is not.}

\subsection{Generalization of Lasserre's hierarchy}
\label{ssec:generalization}

Now we introduce our generalization of Lasserre's hierarchy for 
compact topological packing graphs.

Before we go into the technical details we like to comment on the
choice of spaces in our generalization: In Lasserre's hierarchy for
finite graphs the optimization variable $y$ lies in the cone\footnote{In this paper cones are always assumed to be convex.}
$\R^{I_{2t}}_{\geq 0}$. One might try to use the same cone when
$I_{2t}$ is uncountable. But then there are too many variables and it
is impossible to express the objective function. At the other extreme
one might try to restrict this cone to finitely (or countably)
supported vectors. But then we do not know how to develop a duality
theory like the one in Section~\ref{sec:duality}. A duality theory is
important for concrete computations: Minimization problems can be used
to derive upper bounds rigorously.  We use a cone of Borel
measures where we have ``one degree of freedom'' for every open set.

In Section~\ref{sec:topology} we use the topology of $V$ to equip the
set $I_t$, consisting of the independent sets which have at most $t$
elements, with a Hausdorff topology. There we also use the topological packing graph condition to show that $I_t$ is compact. 

Let $\Ccal(I_{2t})$ be
the set of continuous real-valued functions on $I_{2t}$.  By the Riesz
representation theorem (see e.g.\ \cite[Chapter
2.2]{BergChristensenRessel1984}) the topological dual of
$\Ccal(I_{2t})$, where the topology is defined by the supremum norm,
can be identified with the space $\mathcal{M}(I_{2t})$ of signed Radon
measures. A \emph{signed Radon measure} is the difference of two Radon
measures, where a \emph{Radon measure} $\nu$ is a locally finite
measure on the Borel algebra satisfying \emph{inner regularity}:
$\nu(B) = \sup \{ \nu(C) : C \subseteq B, \, C \text{ compact}\}$ for
each Borel set $B$. Nonnegative functions in $\Ccal(I_{2t})$ form the
cone $\Ccal(I_{2t})_{\geq 0}$. Its conic dual $(\Ccal(I_{2t})_{\geq 0})^*$ is the cone of \emph{positive Radon measures}
\[
\mathcal{M}(I_{2t})_{\geq 0} = \{\lambda \in
\mathcal{M}(I_{2t}) : \lambda(f) \geq 0 \text{ for all } f \in \Ccal(I_{2t})_{\geq 0}\}.
\]

Denote by $\Ccal(I_t \times I_t)_{\text{sym}}$ the space of
\emph{symmetric kernels}, which are the continuous functions $K \colon
I_t \times I_t \to \R$ such that
\[
K(J, J') = K(J', J) \text{ for all }  J,J' \in I_t. 
\]
We say that a symmetric kernel $K$ is
\emph{positive definite} if
\[
(K(J_i, J_j))_{i,j=1}^m \text{ is
 positive semidefinite for all }  m \in \N \text{ and } J_1, \ldots,
J_m \in I_t. 
\]
The positive definite kernels form the cone $\Ccal(I_t \times I_t)_{\succeq
  0}$. The dual of $\Ccal(I_t \times I_t)_{\text{sym}}$ can be
identified with the space of symmetric signed Radon measures
$\mathcal{M}(I_t \times I_t)_{\text{sym}}$. Here a signed Radon
measure $\mu \in \mathcal{M}(I_t \times I_t)$ is \emph{symmetric}
if
\[
\mu(E \times E') = \mu(E'
\times E) \text{ for all Borel sets } E \text{ and } E'. 
\]
We say that a measure $\mu \in \mathcal{M}(I_t \times
I_t)_{\text{sym}}$ is \emph{positive definite} if it lies in the
dual cone $\mathcal{M}(I_t \times I_t)_{\succeq 0} = (\Ccal(I_t
\times I_t)_{\succeq 0})^*$.

\medskip

Now we are ready to define our generalization:

\begin{itemize}
\item[---] The optimization variable is $\lambda \in \mathcal{M}(I_{2t})_{\geq 0}$. \\[-1.7ex]
\item[---] The objective
function evaluates $\lambda$ at $I_{=1}$, where in general,
\[
I_{=t} = \{ S \in I_t : |S| = t\},
\]
and so when $t = 1$ we simply deal with all vertices, as singleton
sets. This is similar to the objective function $\sum_{x \in V}
y_{\{x\}}$ in Lasserre's hierarchy for finite graphs. \\[-1.7ex]
\item[---] The normalization condition reads $\lambda(\{\emptyset\}) = 1$. \\[-1.7ex]
\item[---] For generalizing the moment matrix condition ``$M_t(y)$ is
  positive semidefinite'' we use a dual approach. Let $T_t$ be the operator such that for all vectors $y$ and all matrices $Y$ we have $\langle M_t(y), Y \rangle_1 = \langle y, T_tY \rangle_2$, where $\langle \cdot, \cdot\rangle_1$ is the trace inner product of matrices and $\langle \cdot, \cdot\rangle_2$ is standard vector inner product. Instead of directly generalizing the operator $M_t$, we will dualize the following generalization of the operator $T_t$:
\[
A_t \colon \Ccal(I_t \times I_t)_{\text{sym}} \to
  \Ccal(I_{2t}) \; \text{ by } \; A_tK(S) = \sum_{J, J' \in I_t : J \cup J' = S} K(J, J').
\]
We have $\|A_t K\|_{\infty} \leq 2^{2t} \|K\|_{\infty}$, so $A_t$ is
bounded and hence continuous. Thus there exists the adjoint $A_t^*
\colon \mathcal{M}(I_{2t}) \to \mathcal{M}(I_t \times
I_t)_{\text{sym}}$ and the moment matrix condition reads $A_t^*
\lambda \in \mathcal{M}(I_t \times I_t)_{\succeq 0}$.
\end{itemize}

\begin{definition} The $t$-th step of the \emph{generalized hierarchy} is
\[
\mathrm{las}_t(G) = \sup \Big\{ \lambda(I_{=1}) : \lambda \in \mathcal{M}(I_{2t})_{\geq 0},\;
\lambda(\{\emptyset\}) = 1,\;
A_t^* \lambda \in \mathcal{M}(I_t \times I_t)_{\succeq 0} \Big\}.
\]
\end{definition}

Clearly, we have a nonincreasing chain 
\begin{equation}
\label{eq:chain}
\text{las}_{1}(G) \geq \text{las}_{2}(G) \geq \ldots
 \geq \text{las}_{\alpha(G)}(G) = \text{las}_{\alpha(G)+1}(G) = \ldots
\end{equation}
which stabilizes after $\alpha(G)$ steps, and specializes to the original hierarchy if
$G$ is a finite graph. Each step gives an upper bound for
$\alpha(G)$ because for every independent set $S$ the measure
\[
\lambda = \sum_{R \in I_{2t} : R
  \subseteq S} \delta_R, \quad \text{where } \delta_R \text{ is the
  delta measure at } R,
\]
is a feasible solution for $\text{las}_{t}(G)$ with objective value
$|S|$. To see this we note that $\lambda(\{\emptyset\}) = 1$, and for
any $K \in \Ccal(I_t \times I_t)_{\succeq 0}$ we have
\begin{align*}
\langle K, A_t^*\lambda\rangle = \langle A_t K, \lambda \rangle &= \sum_{R \in I_{2t} : R \subseteq S} \; \sum_{J, J' \in I_t : J \cup J' = R} K(J, J') \\
&= \sum_{J, J' \in I_t : J, J' \subseteq S} K(J, J') \geq 0.
\end{align*}

In Section~\ref{sec:duality} we consider the dual program of
$\mathrm{las}_t(G)$, which is
\begin{align*}
\mathrm{las}_t(G)^* = \inf \Big\{K(\emptyset, \emptyset) : \;& K \in \Ccal(I_t \times I_t)_{\succeq 0},\\[-0.5em]
& A_tK(S) \leq -1_{I_{=1}}(S) \text{ for } S \in I_{2t} \setminus
\{\emptyset\}\Big\},
\end{align*}
and we show that strong duality holds in every step:

\begin{theorem}
\label{thm:strong duality}
  Let $G$ be a compact topological packing graph.  For every $t \in
  \N$ we have $\mathrm{las}_t(G) = \mathrm{las}_t(G)^*$, and if
  $\mathrm{las}_t(G)$ is finite\footnote{We show this in Remark~\ref{rem:bounded}.}, then the optimum in
  $\mathrm{las}_t(G)$ is attained.
\end{theorem}

In Section~\ref{sec:finite convergence} we show that the
chain~\eqref{eq:chain} converges to the independence number:

\begin{theorem} 
\label{thm:finite convergence}
Let $G$ be a compact topological packing graph. Then,
\[
\mathrm{las}_{\alpha(G)}(G) = \alpha(G).
\]
\end{theorem}

A variation of $\text{las}_t(G)$ can be used to upper bound the
weighted independence number of a weighted compact topological packing
graph $G$ with a continuous weight function $w \colon V \to \R_{\geq
  0}$. We extend~$w$, with the obvious abuse of notation, to a
function $w \colon I_{2t} \to \R_{\geq 0}$ where only singleton sets
have positive weight. It turns out, by Lemma~\ref{lem:cont card}, that
also the extension is continuous. Then we replace the objective
function $\lambda(I_{=1})$ by $\lambda(w)$.

\subsection{Explicit computations in the literature}

Explicit computations of $n$-point bounds have been done in a variety
of situations. The following table provides a guide to the literature:

\medskip

\begin{center}
\scriptsize
\begin{tabular}{@{}llll@{}}
\toprule
\bf Packing problem & \bf $2$-point bound & \bf $3$-point bound & \bf $4$-point bound\\
\midrule
Binary codes & Delsarte \cite{Delsarte1973a} & Schrijver \cite{Schrijver2005} & \parbox{0.15\textwidth}{Gijswijt,\\ Mittelmann,\\ Schrijver \cite{GijswijtMittelmanSchrijver2011}}\\
\midrule
$q$-ary codes & Delsarte \cite{Delsarte1973a} & \parbox{0.17\textwidth}{Gijswijt,\\ Schrijver,\\ Tanaka \cite{GijswijtSchrijverTanaka2006}} & \\
\midrule
Constant weight codes & Delsarte \cite{Delsarte1973a} & \parbox{0.17\textwidth}{Schrijver \cite{Schrijver2005},\\  Regts \cite{Regts2009}} & \\
\midrule
Spherical codes & \parbox{0.17\textwidth}{Delsarte,\\Goethals,\\Seidel \cite{DelsarteGoethalsSeidel1977}} & \parbox{0.16\textwidth}{Bachoc,\\Vallentin \cite{BachocVallentin2008}} &  \\
\midrule
Codes in $\R\mathrm{P}^{n-1}$ & \parbox{0.17\textwidth}{Kabatiansky,\\Levenshtein \cite{KabatianskyLevenstein1978}}& \parbox{0.16\textwidth}{Cohn,\\Woo \cite{CohnWoo2011}} & \\
\midrule
Sphere packings & \parbox{0.17\textwidth}{Cohn,\\ Elkies \cite{CohnElkies2003}}&  \\
\midrule
\parbox{0.22\textwidth}{Binary sphere and\\ spherical cap packings}
& \parbox{0.17\textwidth}{de Laat, \\Oliveira,\\ Vallentin
  \cite{LaatOliveiraVallentin2012}}& \\
\midrule
\parbox{0.22\textwidth}{Congruent copies\\ of a convex body}
& \parbox{0.17\textwidth}{Oliveira,\\ Vallentin
  \cite{OliveiraVallentin2013}}& \\
\bottomrule
\end{tabular}
\end{center}

\medskip

For the first three packing problems in this table one can use
Lasserre's hierarchy for finite graphs. For the last five packing
problems in this table our generalization can be used, where in the
last three cases one has to perform a compactification of the vertex
set first.

We elaborate on the connection between these $n$-point bounds and our
hierarchy in Section~\ref{sec:n-point bounds}. The convergence of the
hierarchy, shows that this approach is in theory capable of solving
any given packing problem in discrete geometry.  One attractive
feature of the hierarchy is that already its first steps give strong
upper bounds as one can see from the papers cited in the table above.

\section{Topology on sets of independent sets}
\label{sec:topology}

Let $G = (V, E)$ be a topological packing graph. In this section we
introduce a topology on $I_{t}$, the set of independent sets having
cardinality at most~$t$.

We equip the direct product $V^t$ with the product topology and the
image of $V^t$ under the map
\[
q \colon (v_1, \ldots, v_t) \mapsto \{v_1, \ldots, v_t\}
\]
with the quotient topology. When we add the empty set to the image we
obtain the collection $\mathrm{sub}_t(V)$ of all subsets of $V$ of
cardinality at most $t$, which obtains its topology from the disjoint
union topology. Compactness of $\mathrm{sub}_t(V)$ follows immediately from
compactness of $V$. Handel~\cite[Proposition 2.7]{Handel2000} shows
that it is Hausdorff.

Given $U_1, \ldots, U_r \subseteq V$, define
\[
(U_1, \ldots, U_r)_t = \{ S \in \mathrm{sub}_t(V) : S \subseteq U_1
\cup \cdots \cup U_r, \; S \cap U_i \neq \emptyset \text{ for } 1 \leq
i \leq r
\}.
\]
Handel \cite{Handel2000} observes 
\[
q^{-1}((U_1, \ldots, U_r)_t) = \bigcup_{\substack{\tau : \{1, \ldots, t\} \to
  \{1, \ldots, r\}\\ \tau \text{ surjective}}} U_{\tau(1)} \times \cdots \times U_{\tau(t)}.
\]
This shows that if the sets $U_i$ are open, then $(U_1, \ldots,
U_r)_t$ is open. In fact, if $\mathcal{B}$ is a base for $V$, then
\[
\mathcal{B}_t = \{(U_1, \ldots, U_r)_t : 1 \leq r \leq t,\; U_1, \ldots, U_r \in \mathcal{B}\}
\]
is a base for $\mathrm{sub}_t(V)$. 
Moreover, if $\{u_1,\ldots,u_r\}$ is an element in an open set $U$ in
$\mathrm{sub}_t(V)$, then there are open neighborhoods $U_i$ of $u_i$
such that the open neighborhood $(U_1, \ldots, U_r)_t$ of $\{u_1,
\ldots, u_r\}$ is a contained in $U$.

\medskip

We now endow $I_t$ with a topology as a subset of
$\mathrm{sub}_t(V)$. Clearly, $I_{=1}$ is homeomorphic to~$V$. It is
also immediate that $I_t$ is Hausdorff. Furthermore, it is compact:

\begin{lemma}
\label{lem:compact I}
Let $G = (V, E)$ be a compact topological packing graph. Then $I_t$
is compact for every $t \in \N$.
\end{lemma}

\begin{proof}
  We will show that $I_t$ is closed, respectively that its complement
  $D_t = \mathrm{sub}_t(V) \setminus I_t$ is open in the compact space
  $\mathrm{sub}_t(V)$.  Let $\{x_1, \ldots, x_r\} \in D_t$ be
  arbitrary. Without loss of generality we may assume that $x_1$ and
  $x_2$ are adjacent. By the topological packing graph condition there
  exists an open clique $U \subseteq V$ containing both $x_1$ and
  $x_2$. Since $V$ is a Hausdorff space there exist disjoint open sets
  $U_1$ and $U_2$ such that $x_1 \in U_1 \subseteq U$ and $x_2 \in U_2
  \subseteq U$. Each set in $(U_1, U_2, V, \ldots, V)_t$ contains at
  least one edge, so $ (U_1, U_2, V, \ldots, V)_t \subseteq D_t.  $
  The set $(U_1, U_2, V, \ldots, V)_t$ is an open neighborhood of
  $\{x_1, \ldots, x_r\}$. Hence, $D_t$ is open. \qed
\end{proof}

If the topology on $V$ comes from a metric, then the topology on
$\mathrm{sub}_t(V)$ is given by the Hausdorff distance, see for
example Borsuk and Ulam \cite{Borsuk1931a}. This indicates that
subsets of nonequal cardinality can be close in the topology on
$\mathrm{sub}_t(V)$. However, in the following lemma, we use the
topological packing graph condition to show that independent sets of
different cardinality are in different connected components of
$I_t$.

\begin{lemma}
\label{lem:cont card}
Let $G = (V, E)$ be a topological packing graph. The map
$I_t \to \N$, $S \mapsto |S|$ is continuous for every $t \in \N$. In
particular, $I_{=t}$ is both open and closed.
\end{lemma}

\begin{proof}
  Let $\{S_\alpha\}$ be a net in $I_t$ converging to $\{x_1, \ldots,
  x_r\} \in I_{t}$, where we assume the $x_i$ to be pairwise
  different. By the topological packing graph condition, there exist
  pairwise disjoint open cliques $U_i$ such that $x_i \in U_i$. The
  set $(U_1, \ldots, U_r)_t$ is open and contains $\{x_1, \ldots,
  x_r\}$. Hence, we eventually have $S_\alpha \in (U_1, \ldots,
  U_r)_t$. Then $|S_\alpha| \geq r$ since the $U_i$ are pairwise
  disjoint and $|S_\alpha| \leq r$ since the $U_i$ are cliques. \qed
\end{proof}

\section{Duality theory of the generalized hierarchy}
\label{sec:duality}

\subsection{A primal-dual pair}
\label{ssec:primal-dual pair}

In this section we derive the dual program of the $t$-th step in our
hierarchy $\mathrm{las}_t(G)$.

\medskip

We want to have a symmetric situation between primal and dual. We
consider the dual pairs $(\Ccal(I_{2t}), \mathcal{M}(I_{2t}))$ and
$(\Ccal(I_t \times I_t)_{\text{sym}}, \mathcal{M}(I_t \times
I_t)_{\text{sym}})$ together with the corresponding nondegenerate
bilinear forms
\[
\langle f, \lambda \rangle = \lambda(f) = \int f(S) \, d\lambda(S)
\quad \text{and} \quad
\langle K, \mu \rangle = \mu(K) = \int K(J,J') \, d\mu(J,J').
\]
We endow the spaces with the weakest topologies compatible with the
pairing: the weak topology on the function spaces and the weak*
topology on the measure spaces. From now on we will always use these
topologies unless explicitly stated otherwise. Because the cones
defined in Section~\ref{ssec:generalization} are closed, it follows
from the bipolar theorem that
\[
(\mathcal{M}(I_{2t})_{\geq 0})^* = \Ccal(I_{2t})_{\geq 0} \quad \text{and} \quad
(\mathcal{M}(I_t \times I_t)_{\succeq 0})^*
= \Ccal(I_t \times I_t)_{\succeq 0}.
\]
Hence, the situation is completely symmetric.

\medskip

Recall that the operator
\[
A_t \colon \Ccal(I_t \times I_t)_{\text{sym}} \to \Ccal(I_{2t}),
\quad A_tK(S) = \sum_{J, J' \in I_t : J \cup J' = S} K(J, J')
\]
is continuous in the norm topologies, so it follows that it is
continuous in the weak topologies. In the next subsection we use that
its adjoint $A_t^*$ is injective:

\begin{lemma}
\label{lem:bounded and surjective}
Let $G = (V, E)$ be a compact topological packing graph.
Then the operator $A_t$ is surjective for every $t \in \N$.
\end{lemma}

\begin{proof}
Let $g$ be a function in $\Ccal(I_{2t})$. The continuity of
\[
u \colon I_t \times I_t \to \mathrm{sub}_{2t}(V),\; (J, J') \mapsto J \cup J'
\]
follows from \cite{Handel2000}. Hence
\[
h \colon u^{-1}(I_{2t}) \to \R,\; (J, J') \mapsto \frac{g(J \cup J')}{A_t \mathbbm{1} (J \cup J')}
\]
is continuous where $\mathbbm{1}$ is the kernel which evaluates to $1$
everywhere.

The set $I_{2t}$ is closed in $\mathrm{sub}_{2t}(V)$,
so the preimage $u^{-1}(I_{2t})$ is closed in $I_t \times I_t$. Since
$I_t \times I_t$ is a compact Hausdorff space there exists, by
Tietze's extension theorem, a function $H \in \Ccal(I_t \times
I_t)$ such that $H(J, J') = h(J, J')$ for all $J, J' \in I_t$. For
each $S \in I_{2t}$ we then have
\begin{align*}
A_tH(S) & \; = \; \sum_{J, J' \in I_t : J \cup J' = S} H(J,J') \; = \;
\sum_{J, J' \in I_t : J \cup J' = S} h(J, J')\\
& \; = \; \frac{1}{A_t\mathbbm{1}(S)} \sum_{J, J' \in I_t : J
  \cup J' = S} g(J \cup J') = g(S).
\end{align*}
\qed
\end{proof}

Using the theory of duality in conic optimization problems, see for
instance Barvinok \cite{Barvinok2002}, we derive the dual hierarchy:
\begin{align*}
\mathrm{las}_t(G)^* = \inf \Big\{K(\emptyset, \emptyset) : \; & K \in \Ccal(I_t \times I_t)_{\succeq 0},\\[-0.5em]
& A_tK(S) \leq -1_{I_{=1}}(S) \text{ for } S \in I_{2t} \setminus
\{\emptyset\}\Big\},
\end{align*}
where one should note that by Lemma~\ref{lem:cont card} the
characteristic function $1_{I_{=1}}$ is continuous. It follows from
weak duality that $\mathrm{las}_t(G) \leq \mathrm{las}_t(G)^*$, and
hence $\mathrm{las}_t(G)^*$ upper bounds the independence number. In
the following lemma we give a simple direct proof.

\begin{lemma}
\label{lem:primal upper bound}
Let $G = (V, E)$ be a compact topological packing graph. Then
\[
\alpha(G) \leq \mathrm{las}_t(G)^*
\]
holds for all $t \in \N$.
\end{lemma}

\begin{proof}
Suppose $K$ is feasible and $L$ is an independent set. Then
\begin{align*}
0 &\leq \sum_{J, J' \in \operatorname{sub}_t(L)} K(J, J') = \sum_{S \in \operatorname{sub}_{2t}(L)} A_tK(S)\\
&= K(\emptyset, \emptyset) + \sum_{x \in L} A_t K(\{x\}) + \sum_{S \in \operatorname{sub}_{2t}(L) \setminus \operatorname{sub}_{1}(L)} A_t K(S) 
\leq K(\emptyset, \emptyset) - |L|.
\end{align*}
\qed
\end{proof}
The hierarchy $\mathrm{las}_t(G)^*$ stabilizes after $\alpha(G)$ steps, because the variables and constraints are the same for each $t \geq \alpha(G)$. By Lemma~\ref{lem:cont card} the set $I_t$ is both open and closed in $I_{t+1}$, which means we can extend a feasible kernel $K$ of $\mathrm{las}_t(G)^*$ by zeros to obtain a feasible solution to $\mathrm{las}_{t+1}(G)^*$ with the same objective value. This shows that the hierarchy is nonincreasing; that is, $\mathrm{las}_{t+1}(G)^* \leq \mathrm{las}_t(G)^*$ for all $t$. These results also follow from  strong duality as discussed next.

\subsection{Strong duality; Proof of Theorem~\ref{thm:strong duality}}
\label{ssec:strong duality}

In this section we prove Theorem~\ref{thm:strong duality}: We have
strong duality between the problems $\mathrm{las}_t(G)$ and
$\mathrm{las}_t(G)^*$. We will show the finiteness of
$\mathrm{las}_t(G)^*$ in Remark~\ref{rem:bounded}.

For proving Theorem~\ref{thm:strong duality} we make use of a closed
cone condition, which for example is explained in
Barvinok~\cite[Chapter IV.7]{Barvinok2002}. For this we have to show
that $\mathrm{las}_t(G)$ has a feasible solution, which we already
know from Section~\ref{ssec:generalization}, and that the cone
\[
K =
\big\{(A_t^*\xi - \mu, \xi(I_{=1})) : \mu \in \mathcal{M}(I_t \times I_t)_{\succeq
  0},\, \xi \in \mathcal{M}(I_{2t})_{\geq 0},\, \xi(\{\emptyset\}) = 0\big\}
\]
is closed in $\mathcal{M}(I_t \times I_t)_{\mathrm{sym}} \times \R$. The above cone is the Minkowski difference of
\[
K_1 = \big\{(A_t^*\xi, \xi(I_{=1})) : \xi \in \mathcal{M}(I_{2t})_{\geq 0},\,
\xi(\{\emptyset\}) = 0\big\}
\]
and
\[
K_2 = \big\{(\mu, 0) : \mu \in \mathcal{M}(I_t \times I_t)_{\succeq
  0}\big\}.
\]
By a theorem of Klee~\cite{Klee1955} and
Dieudonn\'e~\cite{Dieudonne1966} the Minkowski difference $K_1 - K_2$
is closed when the three conditions
\begin{enumerate}
\item[(A)] $K_1 \cap K_2 = \{0\}$,
\item[(B)] $K_1$ and $K_2$ are closed,
\item[(C)] $K_1$ is locally compact.
\end{enumerate}
are satisfied. The fact that $K_2$ is closed follows immediately since
$\mathcal{M}(I_t \times I_t)_{\succeq 0}$ is closed. We now verify the
other conditions\footnote{The proof of this lemma has been updated compared to the published version, fixing a problem found independently by Jan Rolfes and Andrew Salmon.}:

\begin{lemma}
\label{lem:trivial intersection}
$K_1 \cap K_2 = \{0\}$.
\end{lemma}
\begin{proof}
  We will show that $\xi \in \mathcal{M}(I_{2t})_{\geq 0}$ with
  $\xi(\{\emptyset\}) = 0 $ is the zero measure if $A_t^*\xi \in
  \mathcal{M}(I_t \times I_t)_{\succeq 0}$.

\smallskip

Let $f \in \Ccal(I_t \times I_t)_{\mathrm{sym}}$ be given by 
\[
f(J, J') = \begin{cases}
1 & \text{if } J = J' = \emptyset,\\
0 & \text{otherwise}.
\end{cases}
\]
Then 
$
A_t^*\xi(\{(\emptyset, \emptyset)\}) = \langle f, A_t^*\xi \rangle = \langle A_t f, \xi \rangle = \xi(\{\emptyset\}) = 0.
$

\smallskip

For $n \in \Z$ define $g_n \in \Ccal(I_t)$ by
\[
g_n(S) = \begin{cases}
|n| & \text{if } S = \emptyset,\\
1/n & \text{otherwise.}
\end{cases}
\]
Since $g_n \otimes g_n \in \Ccal(I_t \times I_t)_{\succeq 0}$ and
$A_t^*\xi \in \mathcal{M}(I_t \times I_t)_{\succeq 0}$ we have
$A_t^*\xi(g_n \otimes g_n) \geq 0$.  We have that $A_t^*\xi(g_n \otimes g_n)$ equates to
\[
n^2 A_t^*\xi\big(\{(\emptyset,
\emptyset)\}\big) \; + \; \frac{1}{n^2} A_t^*\xi \big(I_t \setminus \{\emptyset\}
\times I_t \setminus \{\emptyset\}\big) \; + \; 2\operatorname{sign}(n)
A_t^*\xi\big (\{\emptyset\} \times I_t \setminus \{\emptyset\} \big).
\]
The first term is zero, so the sum of the last two terms is
nonnegative for each $n$. By letting $n$ tend to plus and minus
infinity we see that $A_t^*\xi(\{\emptyset\} \times I_t 
\setminus \{\emptyset\}) = 0$.

\smallskip

Define $h \in \Ccal(I_t \times I_t)_{\text{sym}}$ by
\[
h(J, J') = \begin{cases}
1 & \text{if } J = \emptyset \text{ and } J' = \emptyset,\\
1/2 & \text{if } J = \emptyset \text{ xor } J' = \emptyset,\\
0 & \text{otherwise}.
\end{cases}
\]
Since $\xi \geq 0$ and
\[
\xi(I_{t}) = \langle A_t h, \xi \rangle = \langle h, A_t^*\xi \rangle = A_t^*\xi(\{(\emptyset, \emptyset)\}) + A_t^*\xi(\{\emptyset\} \times I_t \setminus \{\emptyset\}) = 0, 
\]
we have $\xi|_{I_t} = 0$.

If $V$ is a sufficiently small open set in $I_t$, then the union of two distinct sets in $V$ cannot be independent. This shows that  
\[
A_t(1_V \times 1_V)(S) = \sum_{J,J' \in I_t : J \cup J' = S} 1_V(J) 1_V(J') = 0
\]
whenever $S \in I_{2t} \setminus I_t$.

Let $S$ and $S'$ be arbitrary elements in $I_t$, and let $U$ and $U'$ be small open neighborhoods around $S$ and $S'$.

For $s = \pm 1$ we have
\begin{align*}
0 &\leq A_t^*\xi((1_U + s 1_{U'}) \otimes (1_U + s 1_{U'}))\\
&= A_t^*\xi(U \times U) + A_t^*\xi(U' \times U') + 2s A_t^*\xi(U \times U'),
\end{align*}
where the inequality follows because Urysohn's lemma says that $1_U + s 1_{U'}$ can be approximated arbitrarily well by continuous functions. 
Since $\xi|_{I_t} = 0$, we have 
\[
A_t^*\xi(1_U \times 1_U) = A_t^*\xi(1_{U'} \times 1_{U'}) =0
\]
for $U$ and $U'$ small enough.
This shows $2s A_t^*\xi(U \times U') \geq 0$, and since $s =\pm 1$ we have $A_t^*\xi(U \times U') = 0$. Since $S$ and $S'$ are arbitrary this shows $A_t^*\xi = 0$, and since $A_t^*$ is injective, we have $\xi = 0$.
\qed
\end{proof}

\begin{remark}
  The set $I_{2t}$ is a subset of the power set $2^V$. A power set is
  a \emph{monoid} with the associative binary operation $\cup$ and
  unit element $\emptyset$. Monoids have sufficient structure for
  defining functions of \emph{positive type}, which in this case are
  functions $f \colon 2^V \to \R$ for which the matrices $(f(J_i \cup
  J_j))_{i,j=1}^m$ are positive semidefinite for all $m \in \N$ and
  $J_1, \ldots, J_m \in 2^V$. This monoid is \emph{commutative} (i.e., $J \cup J' = J' \cup J$ for all $J,J' \in 2^V$) and
  \emph{idempotent} (i.e., $J \cup J = J$ for all $J \in
  2^V$), so the matrix
\[
\begin{pmatrix}
f(\emptyset) & f(J)\\
f(J) & f(J)
\end{pmatrix} \quad \text{is positive semidefinite},
\]
and so $0 \leq f(J) \leq f(\emptyset)$ for all $J \in 2^V$ \cite[p. 119]{BergChristensenRessel1984}. In
particular, a function of positive type which vanishes at the unit
element is identically zero. This resembles the situation in the proof
of Lemma~\ref{lem:trivial intersection}. To see this we show that one
can view $\lambda \in \mathcal{M}(I_{2t})$ with $A_t^*\lambda \in
\mathcal{M}(I_t \times I_t)_{\succeq 0}$ as a ``measure of positive
type''. For this we notice that a function $f \colon 2^V \to \R$ is of
positive type if and only if $\sum_{S \in 2^V} f(S) \sum_{J \cup J' =
  S} g(J) g(J') \geq 0$ for all finitely supported functions $g \colon
2^V \to \R$. Going from the monoid $2^V$ to the ``truncated monoid''
$I_{2t}$, and from functions to measures, we have the natural
definition that a measure $\lambda \in \mathcal{M}(I_{2t})$ is of
positive type if $\int A_t(g \otimes g)(S) \, d\lambda(S) \geq 0$ for
all $g \in \Ccal(I_{2t})$, which is the case if and only if
$A_t^*\lambda \in \mathcal{M}(I_t \times I_t)_{\succeq 0}$. Moreover,
if we define a convolution and an involution on $\Ccal(I_{2t})$ by $f
* g = A_t(f \otimes g)$ and $f^* = f$,
respectively, then a measure $\lambda$ is of positive type if and only
if $\lambda(f^* * f) \geq 0$ for all $f \in \Ccal(I_{2t})$. This
agrees with the definition of measures of positive type as given for
instance in \cite[Chapter 6.3]{Folland1995} for locally compact
groups, where a different algebra is used.
\end{remark}

Before we consider condition (C) we need some background: A cone is
\emph{locally compact} if it is locally compact as a topological
space, that is, each point in the cone is contained in a compact
neighborhood relative to the cone. A cone is locally compact if the
origin has a compact neighborhood relative to the cone: For each point
$x$ in the cone and each neighborhood $U$ of the origin there is an $r
> 0$ such that $x \in rU$. A \emph{convex base} $B$ of a cone
$K$ is a convex subset of the cone such that every nonzero $x \in K$ can
be written in a unique way as a positive multiple of an element in $B$. A
cone is \emph{pointed} if it does not contain a line. Now we can state
a theorem of Klee and Dieudonn\'e~\cite[(2.4)]{Klee1955}: A nonempty
pointed cone in a locally convex vector space is closed and
locally compact if and only if it admits a compact convex base.

\begin{lemma}
$K_1$ is closed and locally compact.
\end{lemma}
\begin{proof}
Set
\[
B = \{ \xi \in \mathcal{M}(I_{2t})_{\geq 0} : \langle 1_{I_{2t}}, \xi
\rangle = 1, \, \langle 1_\emptyset, \xi \rangle = 0\}.
\]
The maps
\[
\mathcal{M}(I_{2t}) \to \R, \; \xi \mapsto \langle 1_{I_{2t}}, \xi
\rangle \quad \text{and} \quad \mathcal{M}(I_{2t}) \to \R, \; \xi
\mapsto \langle 1_\emptyset, \xi \rangle
\] 
are continuous, so the preimage of $\{1\}$ under the first map and the
preimage of $\{0\}$ under the second map is closed. Hence, $B$ is
closed in the space of probability measures on $I_{2t}$, which is
compact by the Banach-Alaoglu theorem. So, $B$ is compact as well.

By Lemma~\ref{lem:bounded and surjective} $A_t^*$ is injective, so the
map $\xi \mapsto (A_t^*\xi, \xi(I_{=1}))$ is injective and the image
of $B$ under this map is a compact convex base for $K_1$. Hence, by
Klee, Dieudonn\'e, the cone $K_1$ is closed and locally compact.
\qed
\end{proof}

\begin{remark}
  In this remark we show that for infinite graphs the cone $K_2$ is not locally
  compact, and hence it is important that only one of the two cones is
  required to be locally compact in condition (C). If $V$ is an
  infinite set, then so is $I_t$, which means that $\mathcal{M}(I_t)$
  is an infinite dimensional (Hausdorff) topological vector space
  which is therefore not locally compact. The Banach-Alaoglu theorem
  says that the closed ball of radius $r$ centered about the origin in
  $\mathcal{M}(I_t)$ is compact. This means that it cannot be a
  neighborhood of the origin. Thus, for each $r > 0$ there exists a
  net $\{\lambda_\beta\} \subseteq \mathcal{M}(I_t)$ converging to the
  origin, such that $\|\lambda_\beta\| = r$ for all $\beta$.

  Let $f \in \Ccal(I_t \times I_t)_\mathrm{sym}$ and $\epsilon >
  0$. The set
\[
\mathrm{span}\{c\, g \otimes g : c \in \R, \, g \in \Ccal(I_t)\}
\]
is a point separating and nowhere vanishing subalgebra of $\Ccal(I_t
\times I_t)_\mathrm{sym}$, so it follows from the Stone-Weierstrass
theorem that it is dense in the uniform topology. This means that
there exists a function $\tilde{f} = \sum_{i=1}^m c_i g_i \otimes g_i$
such that $\|\tilde{f} - f\|_\infty \leq \epsilon / r^2$. Then,
\begin{align*}
|\lambda_\beta \otimes \lambda_\beta(f)| 
&\leq | \lambda_\beta \otimes \lambda_\beta(f) - \lambda_\beta \otimes \lambda_\beta(\tilde{f})| + |\lambda_\beta \otimes \lambda_\beta(\tilde{f})|\\
&\leq \| \lambda_\beta \otimes \lambda_\beta \| \|f - \tilde{f}\|_\infty + \sum_{i=1}^m c_i  \lambda_\beta(g_i)^2 \to \epsilon.
\end{align*}
So, the net $\{\lambda_\beta \otimes \lambda_\beta\}$ in
$\mathcal{M}(I_t \times I_t)_{\succeq 0}$, which satisfies
$\|\lambda_\beta \otimes \lambda_\beta\|=r^2$ for each $\beta$,
converges to the origin. Therefore, none of the closed balls centered
about the origin is a neighborhood of the origin in $\mathcal{M}(I_t
\times I_t)_{\succeq 0}$. Since compact sets are bounded, this means
that the origin does not have a compact neighborhood in
$\mathcal{M}(I_t \times I_t)_{\succeq 0}$, so this cone is not locally
compact and neither is $K_2$.
\end{remark}

\section{Convergence to the independence number; Proof of
  Theorem~\ref{thm:finite convergence}}
\label{sec:finite convergence}

In this section we prove Theorem~\ref{thm:finite convergence}: The
chain~\eqref{eq:chain} converges to the independence
number~$\alpha(G)$. 

Our proof can be seen as an infinite-dimensional version of Laurent's
proof of the convergence of the hierarchy for finite
graphs~$G = (V,E)$. In~\cite{Laurent2003a} she makes use of the fact that the cone
of positive semidefinite moment matrices where rows and columns are
indexed by the power set $2^V$ is a simplicial polyhedral cone; an
observation due to Lindstr\"om \cite{Lindstroem1969a} and Wilf
\cite{Wilf1968a}. More specifically, 
\begin{equation}
\label{eq:momentcone}
\{M \in \R^{2^V \times 2^V} \hspace{-0.1em}: M \succeq 0, M \text{ is a moment matrix}\} = \cone\{\chi_S \chi_S^{\sf T} : S \subseteq V\},
\end{equation}
where a moment matrix $M$ is a matrix where the entry $M_{J,J'}$ only
depends on the union $J \cup J'$ and where the vector $\chi_S \in
\R^{2^V}$ is defined componentwise by
\[
\chi_S(R) = 
\begin{cases}
1 & \text{if $R \subseteq S$},\\
0 & \text{otherwise.}
\end{cases}
\]

The proof of~\eqref{eq:momentcone} uses the inclusion-exclusion
principle. In our proof the following form of the
inclusion-exclusion principle will be crucial: Given finite sets $A$
and $C$,
\begin{align*}
\sum_{B : A \subseteq B \subseteq C} (-1)^{|B|} &= (-1)^{|A|} \sum_{B \subseteq C\setminus A} (-1)^{|B|}\\
&= (-1)^{|A|} \sum_{i=0}^{|C \setminus A|} \binom{|C \setminus A|}{i} 1^{|C \setminus A| - i} (-1)^i\\
&= (-1)^{|A|} (1 - 1)^{|C \setminus A|} = \begin{cases}
(-1)^{|A|}& \text{if } A = C,\\
0& \text{otherwise.}
\end{cases}
\end{align*}

In our proof we are also faced with two analytical difficuties because we consider infinite graphs: 1. The
cone $\{A_t^* \lambda : \lambda \in \mathcal{M}(I_{2t})\} \cap
\mathcal{M}(I_t \times I_t)_{\succeq 0}$ is not finitely
  generated. 2. Also the power set $2^V$ is too large.

  The second problem we solve by considering the set $I =
  I_{\alpha(G)}$ instead of $2^V$. In fact, already when we defined
  the hierarchy we used measures on independent sets instead of
  measures on subsets of the vertices.

  The first problem we solve by using weak vector valued integrals
  (as discussed in for instance \cite[Appendix 3]{Folland1995})
  instead of finite conic combinations: Let $\tau \in \mathcal{M}(I)$
  and $\nu_S \in \mathcal{M}(I)$ so that $S \mapsto \nu_S$ is a
  continuous map from $I$ to $\mathcal{M}(I)$ with $\sup_{S \in I}
  \|\nu_S\| < \infty$. Then $f \mapsto \int \nu_S(f) \, d\tau(S)$ is a
  bounded linear map on $\Ccal(I)$, and hence defines a unique signed
  Radon measure $\nu$ on $I$ which we denote by $\nu = \int \nu_S \,
  d\tau(S)$. The point measures
\[
\delta_S \quad \text{and} \quad \chi_R = \sum_{Q
    \subseteq R} \delta_Q
\]
which we will use in the next proposition satisfy the above
conditions, so we can use them as integrants in vector valued
integrals.

Now the proof of Theorem~\ref{thm:finite convergence} will follow
immediately from the following proposition.

\begin{proposition}
\label{prop:signed measure}
Let $G$ be a compact topological packing graph and suppose $\lambda$
is feasible for $\mathrm{las}_{\alpha(G)}(G)$. Then there exists a
unique probability measure 
\[
\sigma \in \mathcal P(I) = \{ \lambda \in
\mathcal{M}(I)_{\geq 0} : \|\lambda\| = 1 \}
\]
such that
\[
\lambda =  \int \chi_R\, d\sigma(R).
\] 
\end{proposition}

\begin{proof}
\textsl{Existence:}
We have 
\[
\lambda = \int \delta_S \,d\lambda(S) = \int \sum_{R \subseteq S} (-1)^{|S \setminus R|} \chi_R \, d\lambda(S),
\]
because by the inclusion-exclusion principle 
\[
\sum_{R \subseteq S} (-1)^{|S \setminus R|} \chi_R = \sum_{R \subseteq S} (-1)^{|S \setminus R|} \sum_{Q \subseteq R} \delta_Q = \sum_{Q \subseteq S} \delta_Q \sum_{R : Q \subseteq R \subseteq S} (-1)^{|S \setminus R|} = \delta_S.
\]
The image of $f \in \Ccal(I)$ under the linear map
\[
\Ccal(I) \to \R, \; f \mapsto \int \sum_{R \subseteq S} (-1)^{|S \setminus R|} f(R)\, d\lambda(S)
\]
has norm at most $2^{\alpha(G)} \|\lambda\| \|f\|_\infty$, so the above linear functional is bounded and hence defines a signed Radon measure $\sigma$ on $I$. Then 
\[
\int \chi_R(f)\, d\sigma(R) = \int \sum_{R \subseteq S} (-1)^{|S \setminus R|} \chi_R(f)\, d\lambda(S) = \lambda(f),
\]
for each $f \in \Ccal(I)$, so $\lambda =  \int \chi_R\, d\sigma(R)$.

\medskip 

\textsl{Uniqueness:} If $\sigma' \in \mathcal M(I_{2t})$ is another measure such that $\lambda = \int \chi_R\, d\sigma'(R)$, then $\int \chi_R \,d(\sigma - \sigma')(R) = 0$. Evaluating the above measure at a Borel set $L \subseteq I_{=t}$ with $t = \alpha(G)$ gives 
\[0 = \int \chi_R(L) \,d(\sigma - \sigma')(R) = (\sigma - \sigma')(L),\] so $\sigma|_{I_{=t}} = \sigma'|_{I_{=t}}$. Repeating this argument for $t = \alpha(G)-1, \ldots, 1, 0$ shows $\sigma = \sigma'$, which shows that $\sigma$ is unique.

\medskip 

\textsl{Positivity:} Let $g \in \Ccal(I)_{\geq 0}$ be arbitrary and define $f \in \Ccal(I)$ by
\[
f(Q) = \sum_{P \subseteq Q} (-1)^{|Q \setminus P|} \sqrt{g(P)},
\]
so that
\begin{align*}
\sum_{Q \subseteq R} f(Q)
&= \sum_{Q \subseteq R} \sum_{P \subseteq Q} (-1)^{|Q \setminus P|} \sqrt{g(P)}\\
&= \sum_{P \subseteq R} (-1)^{|P|} \sqrt{g(P)} \sum_{Q : P \subseteq Q \subseteq R} (-1)^{|Q|}
= \sqrt{g(R)}.
\end{align*}
We have
\[
0 \leq \langle f \otimes f, A_{\alpha(G)}^* \lambda \rangle = \langle A_{\alpha(G)} f \otimes f, \lambda \rangle,
\]
and since $\lambda = \int \chi_R \, d\sigma(R)$, the right hand side above is equal to
\[
\int \sum_{Q \subseteq R} A_{\alpha(G)}(f \otimes f)(Q) \,d\sigma(R).
\]
Since we are in the final step of the hierarchy, we have that $A_{\alpha(G)}(f \otimes f)(Q)$ can be written as $\sum_{J \cup J' = Q}f(J)f(J'),$ so the above equals
\[
\int \sum_{Q \subseteq R}\; \sum_{J\cup J' = Q} f(J)f(J')\,d\sigma(R)
= \int \left(\sum_{Q \subseteq R} f(Q)\right)^2 \, d\sigma(R) =
\int g(R) \, d\sigma(R),
\]
which shows that $\sigma$ is a positive measure.

\medskip

\textsl{Normalization:} $\sigma$ is a probability measure, because 
\[
1 = \lambda(\{\emptyset\}) = \int \chi_S(\{\emptyset\})\, d\sigma(S) = \|\sigma\|.
\]
\qed
\end{proof}

\begin{proposition}
Let $G$ be a compact topological packing graph. Then the extreme points of the feasible region of $\mathrm{las}_{\alpha(G)}(G)$ are precisely the measures $\chi_R$ with $R \in I$.
\end{proposition}

\begin{proof}
If $\sigma \in \mathcal{P}(I)$ and $\lambda = \int \chi_R \,d\sigma(R)$, then
\[
\lambda(\{\emptyset\}) = \int \chi_R(\{\emptyset\})\, d\sigma(R) = 1, 
\]
and for each $K \in \Ccal(I \times I)_{\succeq 0}$ we have
\[
\langle K, A_{\alpha(G)}^* \lambda \rangle = \int \chi_R(A_{\alpha(G)} K) \,d\sigma(R) = \int \sum_{J, J' \subseteq R} K(J, J') \, d\sigma(R) \geq 0,
\]
so $\lambda$ is feasible for $\mathrm{las}_{\alpha(G)}(G)$.
So we have the surjective linear map
\[
L \colon \mathcal{P}(I) \to \mathcal F, \; \sigma \mapsto \int \chi_R \,d\sigma(R),
\]
where $\mathcal F$ denotes the feasible set of $\mathrm{las}_{\alpha(G)}(G)$.
By Proposition~\ref{prop:signed measure} the map $L$ is also injective. This means that 
\[
\mathrm{ex}(\mathcal F) = \mathrm{ex}(L(\mathcal{P}(I))) = L(\mathrm{ex}(\mathcal{P}(I)))
\]
and since $\mathrm{ex}(\mathcal{P}(I)) = \{\delta_S : S \in I\}$ (see
for instance Barvinok \cite[Proposition 8.4]{Barvinok2002}), the right hand side above is equal to $ L(\{\delta_S : S \in I\}) = \{ \chi_R : R \in I \}$.
\qed
\end{proof}

\begin{theopargself}
\begin{proof}[ of Theorem~\ref{thm:finite convergence} ]
Let $\lambda$ be feasible for $\mathrm{las}_{\alpha(G)}(G)$. By Proposition~\ref{prop:signed measure} there exists a probability measure $\sigma$ on $I$ such that $\lambda = \int \chi_S \, d\sigma(S)$. Substituting this integral for $\lambda$ in the definition of $\mathrm{las}_{\alpha(G)}(G)$ gives
\[
\mathrm{las}_{\alpha(G)}(G) \leq \max  \Big\{ \int \underbrace{\chi_R(I_{=1})}_{|R|} d\sigma(R) : \sigma \in \mathcal{P}(I) \Big\} = \alpha(G), 
\]
and since we already know that $\mathrm{las}_{\alpha(G)}(G) \geq \alpha(G)$, this completes the proof. \qed
\end{proof}
\end{theopargself}

\section{Two and three-point bounds}
\label{sec:n-point bounds}

\subsection{Two-point bounds}

The Lov\'asz $\vartheta$-number is a two-point bound originally
defined for finite graphs. Bachoc, Nebe, Oliveira, and
Vallentin~\cite{BachocNebeOliveiraVallentin2009} generalized this to
the spherical code graph, and they showed that it is equivalent to the
linear programming bound of Delsarte, Goethals, and
Seidel~\cite{DelsarteGoethalsSeidel1977}. The following generalization
of the $\vartheta'$-number for compact topological packing graphs $G$
is natural:
\begin{align*}
\vartheta'(G)^* = \inf \Big\{ a : \;& a \in \R, \, F \in \Ccal(V \times V)_{\succeq 0}, \\[-0.4ex]
& F(x, x) \leq a - 1\text{ for } x \in V,\\[-0.6ex]
& F(x, y) \leq -1 \text{ for } \{x, y\} \in I_{=2} \Big\}.
\end{align*}

\begin{lemma}
\label{lem:theta is feasible}
Let $G$ be a compact topological packing graph. Then $\vartheta'(G)^*$
has a feasible solution.
\end{lemma}

For finite graphs one can show $\vartheta'(G)^*$ admits a feasible
solution by selecting a matrix $F$ with $F(x, y) = -1$ for $\{x, y\}
\in I_{=2}$ and the diagonal of $F$ large enough so as to make it
diagonally dominant and hence positive semidefinite. For infinite
graphs it is not clear how to adapt this argument, so we use a different
approach.

\begin{theopargself}
\begin{proof}[of Lemma~\ref{lem:theta is feasible} ]
  By the topological packing graph condition there is for each $x \in
  V$ an open clique $C_x$ containing $x$. Since $V$ is a compact
  Hausdorff space, it is a normal space, so there exists an open
  neighborhood $U_x$ of $x$ such that its closure does not intersect
  $V \setminus C_x$. By compactness there exists an $S \subseteq V$
  such that $\{U_x : x \in S\}$ is a finite open cover of $V$.  By
  Urysohn's lemma there is a function $f_x \in \Ccal(V)$ such that
\[
f_x(y) \begin{cases}
= |S| & \text{ if } y \in U_x,\\
\in [-1, |S|] & \text{ if } y \in C_x\setminus U_x,\\
= -1 & \text{ if } y \in V \setminus C_x.     
\end{cases}
\]
Define
\[
F \in \Ccal(V \times V)_{\succeq 0} \text{ by } F = \sum_{x \in S}
f_x \otimes f_x,\; \text{ and } a = |S|^3+1.
\]
Then, 
\[
F(y, y) = \sum_{x \in S} f_x(y)^2 \leq
|S|^3 = a - 1 \text{ for all } y \in V.
\] 
Moreover, if $\{y, y'\} \in
I_{=2}$, then at most one of $y$ and $y'$ lies in $C_x$ for every given
$x \in S$. So, $f_x(y)f_x(y') = -|S|$ if either $y$ or $y'$ lies in $U_x$
and $f_x(y)f_x(y') \leq 1$ if neither $y$ nor $y'$ lies in
$U_x$. Hence, $F(y, y') \leq -1$ for all $\{y, y'\} \in I_{=2}$, and
it follows that $(a, F)$ is feasible for $\vartheta'(G)^*$.
\qed
\end{proof}
\end{theopargself}

Now we show that the first step of our hierarchy equals the
$\vartheta'$-number for compact topological packing graphs, as it is
known for finite graphs.

\begin{theorem}
Let $G$ be a compact topological packing graph. Then 
\[
\mathrm{las}_1(G)^* = \vartheta'(G)^*.
\]
\end{theorem}

We prove this theorem by Lemma~\ref{lem:leq} and
Lemma~\ref{lem:geq}. We first show the easy inequality.

\begin{lemma}
\label{lem:leq}
$\mathrm{las}_1(G)^* \leq \vartheta'(G)^*.$
\end{lemma}

\begin{proof}
Assume $(a, F)$ is feasible for $\vartheta'(G)^*$ and define $K \in \Ccal(I_1 \times I_1)_{\text{sym}}$ by 
\begin{align*}
& K(\emptyset, \emptyset) = a,\\ 
& K(\emptyset, \{x\}) = K(\{x\}, \emptyset) = -1 \text{ for } x \in V,\\
& K(\{x\}, \{y\}) = (F(x, y) + 1)/a \text{ for } x, y \in V.
\end{align*}
To show that $K$ is positive definite we show that the matrix
$(K(J_i, J_j))_{i, j = 1}^m$ is positive semidefinite for all $m \in
\N$ and $J_1, \ldots, J_m \in I_1$ pairwise different. If none of
the $J_i$'s is empty, then it follows directly. Otherwise we may assume
that there are $x_2, \ldots, x_m \in V$ such that $J_1 = \emptyset$ and
$J_i = \{x_i\}$ for $i = 2, \ldots, m$. We have
\[
\Big( K(J_i, J_j) - K(J_i, J_1) K(J_1, J_1)^{-1} K(J_1, J_j) \Big)_{i, j = 2}^m = a^{-1} \big(F(x_i, x_j)\big)_{i, j = 2}^m,
\]
so by the Schur complement $\big(K(J_i, J_j)\big)_{i, j = 1}^m$ is positive semidefinite.

For $x \in V$ we have 
\[
A_1K(\{x\}) = K(\{x\}, \{x\}) + K(\{x\}, \emptyset) + K(\emptyset, \{x\}) = (F(x, x) + 1)/a - 2 \leq -1,
\] 
and for $\{x,y\} \in I_{=2}$ we have 
\begin{align*}
A_1K(\{x,y\}) 
&= K(\{x\}, \{y\}) + K(\{y\}, \{x\})\\
&= (F(x, y) + 1)/a + (F(y, x) + 1)/a \leq 0. 
\end{align*}
So $K$ is feasible for $\mathrm{las}_t(G)^*$ and since $K(\emptyset, \emptyset) = a$ we have $\mathrm{las}_t(G)^* \leq \vartheta'(G)^*$.
\qed
\end{proof}

\begin{remark}
\label{rem:bounded}
  From this lemma we can see that for each $t \in \mathbb{N}$ the optimization
  problem $\mathrm{las}_t(G)^*$ has a feasible solution and so by
  strong duality the maximum in 
  $\mathrm{las}_t(G)$ is attained: By Lemma~\ref{lem:theta is
    feasible}, $\vartheta'(G)^*$ has a feasible solution, hence by the
  lemma above $\mathrm{las}_1(G)^*$ also has one. Then this can be extended
  trivially to a feasible solution for every
  $\mathrm{las}_t(G)^*$.
\end{remark}

To prove the other inequality we will use the following generalization of
the Schur complement.

\begin{lemma}
\label{lem:schur complement}
Let $X$ be a compact Hausdorff space and let $x_1,\ldots,x_n \in X$ be
elements such that the singletons $\{x_i\}$ are open. Suppose $\mu \in
\mathcal M(X \times X)_\text{sym}$ is such that the matrix $A =
(\mu(\{(x_i, x_j)\}))_{i, j =1}^n$ is positive definite. Denote by
$\mathcal F \subseteq \Ccal(X)$ the set of functions which are zero on
$\{x_1, \ldots, x_n\}$ and for $g \in \mathcal F$ define the vector
$v_g \in \R^n$ by $(v_g)_i = \mu(1_{\{x_i\}} \otimes g)$. Then $\mu$
is positive definite if and only if
\[
\mu(g \otimes g) - v_g^{\sf T} A^{-1} v_g \geq 0 \quad \text{for all} \quad g \in \mathcal F.
\]
\end{lemma}

\begin{proof}
  Mercer's theorem says that a kernel $K \in \Ccal(X \times
  X)_\text{sym}$ is positive definite if and only if there exist
  sequences $(f_i)_i$ and $(\lambda_i)_i$ in $\Ccal(X)$ and $\R_{\geq
    0}$ such that $K(x, y) = \sum_{i=1}^\infty \lambda_i f_i \otimes
  f_i (x, y)$, where convergence is uniform and absolute. It follows
  that $\mu \in \mathcal M(X \times X)_{\succeq 0}$ if and only if
  $\mu(f \otimes f) \geq 0$ for all $f \in \Ccal(X)$. Now we use the
  technique as described in for instance the book by Boyd and Vandenberghe
  \cite[Appendix A.5.5]{BoydVandenberghe2004} and note that the
  measure $\mu$ is positive definite if and only if the function $p
  \colon \R^n \times \mathcal F \to \R$ given by
\begin{align*}
p(r, g) &= \mu( (r_1 1_{\{x_1\}} + \cdots + r_n 1_{\{x_n\}} + g)
\otimes (r_1 1_{\{x_1\}} + \cdots + r_n 1_{\{x_n\}} + g) )\\
&= \mu(g \otimes g) + r^{\sf T} A r + 2r^{\sf T} v_g
\end{align*}
is nonnegative on its domain.  We have $\nabla\!_r\, p(r, g) = 2Ar +
2v_g$, so for fixed $g$, the minimum of $p$ is attained for $r =
-A^{-1}v_g$.  Hence $p$ is nonnegative on its domain if and only if
$\mu(g \otimes g) - v_g^{\sf T} A^{-1} v_g \geq 0$ for all $g \in
\mathcal F$.
\qed
\end{proof}

\begin{lemma}
\label{lem:geq}
$\mathrm{las}_1(G)^* \geq \vartheta'(G)^*.$
\end{lemma}

\begin{proof}
  We will use the duals of $\vartheta'(G)^*$ and
  $\mathrm{las}_1(G)^*$. We derive the dual $\vartheta'(G)$ of
  $\vartheta'(G)^*$ similarly to
  Section~\ref{ssec:primal-dual pair}. We have
\begin{align*}
\vartheta'(G) = 
\sup \Big\{ \eta(I_2 \setminus \{\emptyset\}) : \;& \eta \in \mathcal M(I_2 \setminus \{\emptyset\})_{\geq 0}, \\[-0.5em]
& \eta(I_{=1}) = 1,\; T^* \eta \in \mathcal M(I_{=1} \times I_{=1})_{\succeq 0} \Big\},
\end{align*}
where $T \colon \Ccal(I_{=1} \times I_{=1}) \to \Ccal(I_2 \setminus \{\emptyset\})$ is the operator defined by
\[
TF(S) = \begin{cases}
F(\{x\}, \{x\}) & \text{if $S = \{x\}$},\\
\frac{1}{2}(F(\{x\}, \{y\}) + F(\{y\}, \{x\})) & \text{if $S = \{x, y\}$}.
\end{cases}
\]

Now we prove strong duality: $\vartheta'(G) = \vartheta'(G)^*$ and the
optimum in $\vartheta'(G)$ is attained. Following the approach from
Section~\ref{ssec:strong duality} we first observe that every
probability measure on $I_{=1}$ is feasible for $\vartheta'(G)$. To
complete the proof we show that
\begin{align*}
K = \{(T^*\eta - \nu, \eta(I_2 \setminus \{\emptyset\})) : \; & \nu \in \mathcal M(I_{=1} \times I_{=1})_{\succeq 0}, \\
& \eta \in \mathcal M(I_2 \setminus \{\emptyset\})_{\geq 0}, \, \eta(I_{=1}) = 0\}
\end{align*}
is closed in $\mathcal M(I_{=1} \times I_{=1})_\text{sym} \times \R$. We decompose $K$ as the Minkowski difference of
\[
K_1 = \{(T^*\eta, \eta(I_2 \setminus \{\emptyset\})) : \eta \in \mathcal M(I_2 \setminus \{\emptyset\})_{\geq 0}, \, \eta(I_{=1}) = 0\}
\]
and
\[
K_2 = \{(\nu, 0) : \nu \in \mathcal M(I_{=1} \times I_{=1})_{\succeq 0}\}.
\]
It is immediate that $K_1 \cap K_2 = \{0\}$ and again using the
approach from Section~\ref{ssec:strong duality} we see that $K_1$ and
$K_2$ are closed and that $K_1$ is locally compact. 

Now we show the inequality $\vartheta'(G) \leq \mathrm{las}_1(G)$. Let
$\eta$ be an optimal solution for $\vartheta'(G)$ and define $\lambda
\in \mathcal M(I_2)$ by $\lambda(\{\emptyset\}) = 1$ and
\[
\lambda(L) = \begin{cases}
\vartheta'(G) \eta(L) & \text{if $L$ is a Borel set in $I_{=1}$},\\               
\frac{1}{2}\vartheta'(G) \eta(L) & \text{if $L$ is a Borel set in $I_{=2}$}.
\end{cases}
\]
Then 
\[
\lambda(I_{=1}) = \vartheta'(G)\eta(I_{=1}) = \vartheta'(G). 
\]

To complete the proof we have to show $A_1^* \lambda \in \mathcal
M(I_1 \times I_1)_{\succeq 0}$. We apply our generalized Schur
complement: Let $g \in \Ccal(I_1)$ be a function with $g(\emptyset) =
0$. We have
\[
A_1^* \lambda(g \otimes g) = \vartheta'(G) T^*\eta(g \otimes g).
\]
The symmetric bilinear form $(h, g) \mapsto T^*\eta(h \otimes g)$ is
positive semidefinite because $T^*\eta \in \mathcal M(I_{=1} \times
I_{=1})_{\succeq 0}$, so we can apply the Cauchy-Schwarz inequality
and optimality of $\eta$ to obtain
\[
\vartheta'(G) T^*\eta(g \otimes g) \geq \frac{\vartheta'(G)}{T^*\eta(1_{I_{=1}}\otimes 1_{I_{=1}})} (T^*\eta(1_{I_{=1}} \otimes g))^2 = (T^*\eta(1_{I_{=1}} \otimes g))^2.
\]
In the remainder of this proof we show
\[
T^*\eta(1_{I_{=1}} \otimes g) = \vartheta'(G) \eta(g).
\]
Since
\[
\vartheta'(G) \eta(g) = \lambda(g) = A_1^*\lambda(1_{\emptyset} \otimes g)
\]
the proof is then complete by using the generalized Schur complement,
Lemma~\ref{lem:schur complement}.

Inspired by Schrijver \cite[Theorem 67.10]{Schrijver2003a} we use
Lagrange multipliers. First observe that
\[
T(1_{I_{=1}} \otimes 1_{I_{=1}}) = 1_{I_2 \setminus \{\emptyset\}}
  \quad \text{and} \quad
T^*\eta(1_{I_{=1}} \otimes 1_{I_{=1}}) = \eta(I_2 \setminus \{\emptyset\}).
\]
For $u \in \R^2$ define $g_u \in \Ccal(I_{=1})$
by $g_u = u_1 g + u_2 (1_{I_{=1}}-g)$. For each $u \in \R^2$ with
$\eta(g_u^2) = 1$, the measure $\tilde\eta$ defined by $d\tilde\eta(S)
= T(g_u \otimes g_u)(S) d\eta(S)$ is feasible for $\vartheta'(G)$. So,
if we consider the problem of maximizing $T^*\eta(g_u \otimes g_u)$
over all $u \in \R^2$ for which $\eta(g_u^2) = 1$, then optimality of
$\eta$ implies that an optimal solution is attained for $u = 1$.

It follows that there exists a Lagrange multiplier $c \in \R$ such that
\[
\frac{\partial}{\partial u_i}\bigg|_{u=(1,1)} T^*\eta(g_u \otimes g_u) = c \frac{\partial}{\partial u_i}\bigg|_{u=(1,1)} \eta(g_u^2) \quad \text{for} \quad i=1,2.
\]
Since
\[
T^*\eta(g_u \otimes g_u) = 
u^{\sf T}
\begin{pmatrix} T^*\eta(g \otimes g) & T^*\eta(g \otimes (1_{I_{=1}}-g)) \\ T^*\eta(g \otimes (1_{I_{=1}}-g)) & T^*\eta((1_{I_{=1}}-g) \otimes (1_{I_{=1}}-g))\end{pmatrix}
u
\]
and
\[
\eta(g_u^2) = 
u^{\sf T}
\begin{pmatrix} \eta(g^2) & \eta(g(1_{I_{=1}}-g)) \\ \eta(g(1_{I_{=1}}-g)) & \eta((1_{I_{=1}}-g)^2)\end{pmatrix}
u
\]
we have
\[
T^*\eta(g \otimes 1_{I_{=1}}) = c \eta(g) \quad \text{and} \quad T^*\eta((1_{I_{=1}}-g) \otimes 1_{I_{=1}}) = c \eta(1_{I_{=1}}-g).
\]
By summing the last two equations we see that $c = \vartheta'(G)$, hence
we have the desired equality $T^*\eta(g \otimes 1_{I_{=1}}) = \vartheta'(G) \eta(g)$.
\qed
\end{proof}

\subsection{Three-point bounds}

In this section we modify the $2t$-point bound $\mathrm{las}_t(G)$ to obtain a $2t+1$-point bound for sufficiently symmetric graphs $G$. For the spherical code graph this gives an easy derivation of a variation of the three-point bound given by Bachoc and Vallentin in \cite{BachocVallentin2008}.

Let $G = (V, E)$ be a compact topological packing graph. We are interested in two groups related to $G$. The group of graph automorphisms of $G$ and the
group of homeomorphisms of the topological space $V$. When we endow the
latter group with the compact-open topology, it is a topological group
with a continuous action on $V$; see Arens \cite{Arens1946}. In the
special case when $G$ is a distance graph, as defined in
Section~\ref{ssec:topological packing graphs}, the former group is
contained in the latter. We say that $G$ is \emph{homogeneous} if
there exists a compact subgroup of the group of homeomorphisms which
consists only of graph automorphisms and is such that the action of
$\Gamma$ on $V$ is transitive.

Fix a point $e \in V$. By $G^e$ we denote the induced subgraph of $G$
with vertex set
\[
V^e = \{x \in V : x \neq e \text{ and } \{e, x\} \not \in E\}.
\]
It follows that $G^e$ is also a compact topological packing graph. We have $\alpha(G) \geq 1 + \alpha(G^e)$, and if $G$
is homogeneous, then $\alpha(G) = 1 + \alpha(G^e)$: If $S$ is an independent set of $G$, then there exists a graph automorphism $\gamma$ with $e \in \gamma S$, and $(\gamma S) \setminus \{e\} \subseteq V^e$ is an independent set for $\alpha(G^e)$. So, for
computing an upper bound on the independence number of $G$ we can also
compute $1 + \mathrm{las}_t(G^e)$. This yields a bound which is at
least as good as $\mathrm{las}_t(G)$:
\begin{lemma}
Suppose $G$ is a compact topological packing graph. Then
\[
1 + \mathrm{las}_t(G^e) \leq \mathrm{las}_t(G).
\]
\end{lemma}
\begin{proof}
  We denote the sets of independent sets of $G^e$ by $I_t^e$ and
  $I_{=t}^e$.
  Suppose $\lambda^e$ is feasible for $\text{las}_t(G^e)$. Let
  $\lambda = \delta_e + \lambda^e$. We have $\lambda \geq 0$ and
  $\lambda(\{\emptyset\}) = 1$. Moreover, since $A_t^*\lambda =
  \delta_e \otimes \delta_e + A_t^*\lambda^e$ and $A_t^*\lambda^e \in
  \mathcal{M}(I_t^e \times I_t^e)_{\succeq 0} \subseteq
  \mathcal{M}(I_t \times I_t)_{\succeq 0}$ we have $A_t^*\lambda \in
  \mathcal{M}(I_t \times I_t)_{\succeq 0}$. So $\lambda$ is feasible
  for $\mathrm{las}_t(G)$. We have $1 + \lambda^e(I_{=1}^e) = \lambda(I_{=1})$ which completes the proof.
\qed
\end{proof}

In the handbook chapter
\cite [Theorem 9.15]{BachocGijswijtSchrijverVallentin2010} Bachoc, Gijswijt,
Schrijver, and Vallentin gave a simplified, but computationally
slightly less powerful, variation of the three-point bound given by Bachoc and Vallentin \cite{BachocVallentin2008} for spherical codes. In both
cases the bounds are formulated using the representation theory coming
from the action of the orthogonal group on the unit sphere
$S^{n-1}$. The variation admits a generalization to compact
topological packing graphs wich we can formulate as
\begin{align*}
1 + \inf \Big\{ F(e, e) : \; & F \in \Ccal(V^e \cup \{e\} \times V^e
\cup \{e\})_{\succeq 0}, \\[-0.3ex]
& F(x, x) + F(e, x) + F(x, e) \leq -1 \text{ for } \{e, x\} \in I_{=2},\\[-0.5ex]
& F(x, y) \leq 0 \text{ for } \{e, x, y\} \in I_{=3} \Big\}.
\end{align*}

\begin{proposition}
Suppose $G$ is a compact topological packing graph. Then the optimal value of the optimization problem above equals $1 + \mathrm{las}_1(G^e)^*$.
\end{proposition}

\begin{proof}
Given $F \in \Ccal(V^e \cup \{e\} \times V^e
\cup \{e\})_\text{sym}$ we define $K \in \mathcal C(I_1 \times I_1)_\text{sym}$ by
\begin{align*}
&K(\emptyset, \emptyset) = F(e, e),\\
&K(\emptyset, \{x\}) = K(\{x\}, \emptyset) = F(e, x) \text{ for } \{e, x\} \in I_{=2},\\
&K(\{x\}, \{y\}) = F(x, y) \text{ for } \{e, x, y\} \in I_{=3}.
\end{align*}
The above construction gives a bijection from the feasible region of the above optimization problem onto the feasible region of $\mathrm{las}_1(G^e)$, and since it preserves objective values this completes the proof.
\qed
\end{proof}

\begin{acknowledgements}
We would like to thank Evan DeCorte and Crist\'obal Guzm\'an for very helpful
discussions. We also thank the referee whose suggestions helped to improve the paper.
\end{acknowledgements}

\end{document}